\documentclass[12pt,reqno,oneside]{amsart}
\usepackage{geometry}                
\usepackage{graphicx}
\usepackage{amssymb}
\usepackage{epstopdf}
\usepackage{bm}

\usepackage{geometry}
\usepackage{hyperref}
\usepackage{graphicx}

\usepackage{booktabs} 
\usepackage{array} 
\usepackage{paralist} 
\usepackage{verbatim} 
\usepackage{subfig} 
\usepackage{tabularx}
\usepackage{fancyhdr} 
\usepackage{amsmath,amsfonts,amsthm,mathrsfs,amssymb,cite}
\usepackage[usenames]{color}

\newcommand{\Bx}{\mathbf{x}}
\newcommand{\By}{\mathbf{y}}

\newcommand{\p}{\partial}

\newtheorem{thm}{Theorem}[section]

\newtheorem{lem}{Lemma}[section]

\theoremstyle{definition}

\theoremstyle{remark}

\newtheorem{rem}{Remark}[section]
\numberwithin{equation}{section}

\numberwithin{equation}{section}
\newcounter{saveeqn}

\newcommand{\eqnref}[1]{(\ref {#1})}

\newcommand{\Bz}{\mathbf{z}}

\newcommand{\Acal}{\mathcal{A}}
\newcommand{\Kcal}{\mathcal{K}}

\newcommand{\Scal}{\mathcal{S}}
\newcommand{\Dcal}{\mathcal{D}}

\newcommand{\Ocal}{\mathcal{O}}

\newcommand{\ds}{\displaystyle}

\newcommand{\RR}{\mathbb{R}}

\newcommand{\beq}{\begin{equation}}
\newcommand{\eeq}{\end{equation}}

\DeclareMathAlphabet{\itbf}{OML}{cmm}{b}{it}

\title[Blow-up estimates for electric fields in the quasi-static regime]{Gradient estimates for electric fields with multi-scale inclusions in the quasi-static regime }

\author{Youjun Deng}
\address{School of Mathematics and Statistics, Central South University, Changsha, Hunan, China.}
\email{youjundeng@csu.edu.cn, dengyijun\_001@163.com}

\author{Xiaoping Fang}
\address{College of Science, Hunan University of Commerce, Changsha 410205, China; Key Laboratory of Hunan Province for Statistical Learning and Intelligent Computation, Hunan University of Commerce, Changsha  410205, China}
\email{fxp1222@163.com}

\author{Hongyu Liu}
\address{Department of Mathematics, City University of Hong Kong, Hong Kong SAR, China.}
\email{hongyu.liuip@gmail.com}

\begin{document}
\maketitle

\begin{abstract}
In this paper, we are concerned with the  gradient estimate of the electric field due to two nearly touching dielectric inclusions, which is a central topic in the theory of composite materials. We derive accurate quantitative characterisations of the gradient fields in the transverse electromagnetic case within the quasi-static regime, which clearly indicate the optimal blowup rate or non-blowup of the gradient fields in different scenarios. There are mainly two novelties of our study. First, the sizes of the two material inclusions may be of different scales. Second, we consider our study in the quasi-static regime, whereas most of the existing studies are concerned with the static case.

\medskip

\medskip



\end{abstract}

\section{Introduction}

Stress concentration is a peculiar phenomenon that widely occurs in continuum mechanics. It is a central topic in the theory of composite materials, where the concentration occurs due to the nearly touching of material inclusions that are the building blocks of the composite material. The degree of concentration is characterised by the blowup rate of the gradient of the underlying field. There are extensive studies in the literature on the gradient estimates of the underlying fields due to two nearly touching inclusions. We refer to \cite{LN03,LV00} for related results in general elliptic system,
\cite{ACKLY13,BLL15,BLL17,KLY15,KY19,Yun09} for elastostatics, \cite{AKKY21} for stokes flow problem, and \cite{AKLL07,AKLLZ09,AKL05,BLY09,Yun07, BLY10,KLY14,Lek10,KY09} in electrostatics for optical materials. The gradient estimates depend on the background field as well as the asymptotic parameter $\epsilon$ which signifies the distance between the closely spaced material inclusions. Generically, the optimal blow up rate of the gradient field is of order $1/\sqrt{\epsilon}$ in two dimensions, whereas it is $(\epsilon|\ln\epsilon|)^{-1}$  in three dimensions. In establishing those results, it is usually assumed that the inclusions are of regular size, i.e., the size is of order $\mathcal{O}(1)$ compared to the asymptotic distance parameter $\epsilon\ll 1$. In fact, it is shown in \cite{AKLL07,KLY13} that if the size of the two objects are of the same order as the distance between them, the gradient stays bounded. To our best knowledge, there are few studies on the case that the sizes of the inclusions are of different scales. Moreover, very few results are concerned with the gradient estimates for waves in the frequency regime. There is a major difficulty for the latter case, i.e. the maximum principle fails for the wave system (cf. \cite{BNVa94}).

In this paper, we study the gradient estimate for the electromagnetic field in the transverse model in $\mathbb{R}^2$ due to nearly touching dielectric inclusions. We consider our study in the quasi-static regime, namely the size of the inclusion is smaller than the operating wavelength. Nevertheless, we allow the sizes of the inclusions to be of the same scale or different scales. That is, one inclusion may be of regular size, while the size of the other one can be very large (actually, can be related to the asymptotic parameter $\epsilon$). Geometrically, this means that the curvatures of the nearly touching faces of the two inclusions may be in sharply different scales, say e.g. one is very high while the other is very low (nearly flat). In such a general scenario, we derive an accurate gradient estimate of the electric field, which is contained in \eqref{eq:leessest01} in Theorem \eqref{th:main01}. There are two parts in the asymptotic estimate: the first one accounts for the static effect, whereas the second one accounts for the frequency effect. The static part recovers the known results in the literature if both inclusions are of regular size. It also covers the more general scenario that the two inclusions are of sharply different scales.
It is more interesting to note that the frequency part can induce new blowup phenomena. In fact, even if the static part vanishes, there might still be the blowup phenomenon in certain generic scenarios due to the frequency part. In deriving the new gradient estimate, we develop techniques that combine layer-potential operators with asymptotic analysis and singular decomposition of the wave field.

The rest of the paper is organized as follows. In Section 2, we present the mathematical setup of our study as well as state the main results of the paper. In Section 3, we use layer potential technique to derive the integral representation of the solution as well as the associated asymptotic expansions. The estimates of the nonsingular and singular parts of the gradient fields are established in Sections 4 and 5, respectively.

\section{Mathematical setup and statement of the main results}

In this section, we present the mathematical formulation of the transverse electromagnetic scattering with multi-scale dielectric inclusions. Then we state the main results in this paper, whose proofs shall be postponed to the subsequent sections.

\subsection{Mathematical setup}

Let $B_1$ and $B_2$ be two disks in $\mathbb{R}^2$. Let $\Bz_j\in\mathbb{R}^2$ and $r_j\in\mathbb{R}_+$ be the center and radius of $B_j$, $j=1, 2$, respectively. Define $\epsilon:=\mathrm{dist}(B_1, B_2)$ and suppose $\epsilon\ll 1$. Here, $B_1$ and $B_2$ represent the two dielectric inclusions and they are closely spaced, characterised by the asymptotic distance parameter $\epsilon\in\mathbb{R}_+$. By rigid motions if necessary, we can assume without loss of generality that
\begin{equation}\label{eq:cc1}
\Bz_1=(-r_1-\frac{\epsilon}{2}, 0)\quad \mbox{and} \quad \Bz_2=(r_2+\frac{\epsilon}{2}, 0).
\end{equation}
In what follows, we set
\begin{equation}\label{eq:cc2}
r_1=r_{1,\alpha_1}\epsilon^{\alpha_1}\quad \mbox{and} \quad r_2=r_{2,\alpha_2}\epsilon^{\alpha_2},\quad \alpha_j\in\mathbb{R},\ j=1,2,
\end{equation}
 where $r_{1,\alpha_1}$ and $r_{2,\alpha_2}$ are positive constants that are independent of $\epsilon$. It is pointed out that if one takes $\alpha_1=\alpha_2=0$, then both $B_1$ and $B_2$ are of regular size. It is emphasized that $\alpha_j$ can be negative or positive, respectively corresponding to the high- and low-curvature cases.
 Define
 \begin{equation}\label{eq:cc3}
 \alpha_+=\max(\alpha_1, \alpha_2)\quad \mbox{and}\quad \alpha_-=\min(\alpha_1,\alpha_2).
 \end{equation}

As mentioned earlier, $B_1$ and $B_2$ signify two dielectric inclusions embedded in a uniformly homogeneous medium. The medium parameters are characterised by the electric permittivity $\varepsilon$ and magnetic permeability $\mu$. By normalisation, we assume that $\varepsilon=\mu=1$ in $\mathbb{R}^2\setminus\overline{B_1\cup B_2}$. Let $\varepsilon=\varepsilon_1$ and $\mu=1$ in $B_1\cup B_2$, where $\varepsilon_1\in\mathbb{R}_+$. We consider the transverse magnetic scattering, which is described by the following system (cf. \cite{CDL}):
\beq\label{eq:helm00}
\begin{cases}
\Delta u^* + \omega^2 u^* = 0 \hspace*{3.15cm} \mbox{in} \quad \RR^2\setminus\overline{B_1\cup B_2},\\
\nabla\cdot (\frac{1}{\varepsilon_1}\nabla u^*) +\omega^2 u^*= 0 \hspace*{1.8cm} \mbox{in} \quad B_1\cup B_2,\\
u^*|_+= u^*|_-, \quad \frac{\p u^*}{\p \nu}\Big|_+=  \frac{1}{\varepsilon_1}\frac{\p u^*}{\p \nu}\Big|_-\quad \mbox{on} \quad \p B_1\cup \p B_2,\\
(u^*-u^i)(\Bx) \mbox{ satisfies the Sommerfeld radiation condition,}
\end{cases}
\eeq
where $\omega\in\mathbb{R}_+$ signifies the angular frequency of the wave propagation and, $u^i$ and $u^*$ respectively denote the incident and total wave fields. $u^i$ is an entire solution to $\Delta u^i+ \omega^2 u^i=0$ in $\RR^2$, and one special case is that it is a plane wave of the form $u^i=e^{i\omega\Bx\cdot\mathbf{d}}$, where $\mathbf{d}\in \mathbb{S}^2$ signifies the impinging direction.
By the Sommerfeld radiation condition, we mean that the scattered wave $u^s(\Bx)=(u^*-u^i)(\Bx)$ satisfies
\begin{equation}\label{eq:rad}
\lim_{|\mathbf{x}|\rightarrow +\infty}|\mathbf{x}|^{1/2}\left(\frac{\partial u^s(\mathbf{x})}{\partial |\mathbf{x}|}-i\omega u^s(\mathbf{x})\right)=0.
\end{equation}
Throughout the rest paper, we shall consider $\omega\ll 1$ and $\varepsilon_1=\Ocal(\omega)$.

\subsection{Main gradient estimate and discussion}
We present our main result in this paper as follows:
\begin{thm}\label{th:main01}
Suppose $\omega\cdot\epsilon^{\alpha_-}\ll 1$, and
\beq\label{eq:thuexp01}
u^i=u_0^i+\sum_{j=1}^\infty \omega^j u_j^i,
\eeq
where the functions $u_j^i$, $j=0, 1, 2, \ldots$ are independent of $\omega$. Let $u^*(\Bx)$ be defined in \eqnref{eq:helm00}. Then for any bounded set $\Omega$ containing $\overline{B}_1$ and $\overline{B}_2$, it holds that
\beq\label{eq:leessest01}
\begin{split}
\|\nabla u^*\|_{L^{\infty}(\Omega\setminus\overline{B_1\cup B_2})}\sim &\frac{C_0}{r_-}\epsilon^{\min(\alpha_+,1)/2-1/2}\Big(\p_{\Bx_1} u^i(\mathbf{0})+\frac{1}{\pi}\omega^2|\ln \omega|\int_{B_1\cup B_2} \p_{\Bx_1} u^i \\
&\quad\quad\quad\quad\quad\quad\quad\quad+\Ocal(\omega^2)\Big)+\Ocal(1),
\end{split}
\eeq
where $r_-$ is defined by
\beq\label{eq:thdenr01}
r_-=
\left\{
\begin{array}{ll}
\frac{\alpha_2-\alpha_-}{\alpha_2-\alpha_1}r_{1,\alpha_1}+\frac{\alpha_1-\alpha_-}{\alpha_1-\alpha_2}r_{2,\alpha_2}, & \alpha_1\neq \alpha_2, \\
r_{1,\alpha_1}+r_{2,\alpha_2}, & \alpha_1= \alpha_2,
\end{array}
\right.
\eeq
and $C_0>0$ is the coefficient of the leading order term of $\tau$ defined in \eqnref{eq:deftau01} in what follows.
\end{thm}
\begin{rem}
It is worth mentioning that if $\alpha_1=\alpha_2=0$ then from \eqnref{eq:deftau01}, one has
$$
C_0=\sqrt{2r_{1,0}r_{2,0}(r_{1,0}+r_{2,0})},
$$
and there holds the estimate
\[
\begin{split}
\|\nabla u^*\|_{L^{\infty}(\Omega\setminus\overline{B_1\cup B_2})}\sim &\sqrt{\frac{2r_{1,0}r_{2,0}}{(r_{1,0}+r_{2,0})}}\epsilon^{-1/2}\Big(\p_{\Bx_1} u^i(\mathbf{0})+\frac{1}{\pi}\omega^2|\ln \omega|\int_{B_1\cup B_2} \p_{\Bx_1} u^i \\
&\quad\quad\quad\quad\quad\quad\quad\quad\quad+\Ocal(\omega^2)\Big)+\Ocal(1),
\end{split}
\]
which recovers the blowup estimate for the static case (\cite{AKLL07,BLY09,KLY13}).
\end{rem}
\begin{rem}
It can be seen that if one inclusion is of high curvature, i.e., $\alpha_+>0$ and $\p_{\Bx_1} u^i_0(\mathbf{0})\neq 0$, then the blowup rate is
$\epsilon^{\min(\alpha_+,1)/2-1/2}$, which is less than $\epsilon^{-1/2}$. No blow up occurs in the case that $\alpha_+\geq 1$.
\end{rem}
\begin{rem}
We emphasise that the estimate \eqnref{eq:leessest01} also holds for the low curvature case, i.e., $\alpha_+<0$. In such case, the blowup rate is
$\epsilon^{\alpha_+/2-1/2}$, which is bigger than $\epsilon^{-1/2}$, if $\p_{\Bx_1} u^i_0(\mathbf{0})\neq 0$. Moreover, even if $\p_{\Bx_1} u^i(\mathbf{0})= 0$, one can still have the blowup if $\p_{\Bx_1} u_1^i(\mathbf{0})= 0$ and
$$
-\log_{\epsilon}\omega<\alpha_+< 1-2\log_{\epsilon}\omega,
$$
or
$$
\int_{B_1\cup B_2} \p_{\Bx_1} u_0^i \neq 0,
$$
$\alpha_+$ satisfies
$$
-\log_{\epsilon}\omega<\alpha_+< 1-2\log_{\epsilon}(\omega^2|\ln \omega|).
$$
\end{rem}
\subsection{Key decompositions}
In this subsection, we present the main auxiliary results that we shall derive in order to prove the main result in Theorem~\ref{th:main01}, whose proofs are deferred to the subsequent sections. To estimate the gradient filed of the solution to \eqnref{eq:helm00}, we shall decompose the system into several parts.
We first introduce the following system:
\beq\label{eq:helm01}
\begin{cases}
\Delta u + \omega^2 u = 0 \hspace*{.7cm} \mbox{in} \quad \RR^2\setminus\overline{B_1\cup B_2},\\
u= \lambda_1+\Ocal(\omega^2) \quad \mbox{on} \quad \p B_1,\\
u= \lambda_2+\Ocal(\omega^2)  \quad \mbox{on} \quad \p B_2,\\
(u-u^i)(\Bx) \mbox{ satisfies the Sommerfeld radiation condition,}
\end{cases}
\eeq
where the constants $\lambda_j$, $j=1, 2$ are determined by
\beq
\int_{\p B_j}\p_{\nu} u|_+=\Ocal(\omega^2), \quad j=1, 2,
\eeq
and they are unique up to $\mathcal{O}(\omega^2)$.

We have the following result:
\begin{lem}\label{le:decom01}
Let $u^*$ and $u$ be the solution to system \eqnref{eq:helm00} and \eqnref{eq:helm01}, respectively. Then it holds that
\beq\label{eq:lekeyde01}
\nabla u^*=\nabla u +C\omega+\Ocal(\omega^2) \quad \mbox{in} \quad \RR^2\setminus\overline{B_1\cup B_2},
\eeq
where $C$ is a generic constant that does not depend on $\omega$ and $\epsilon$.
\end{lem}
In what follows, we shall decompose the solution to \eqnref{eq:helm01} into two parts as follows:
\beq\label{eq:decomp01}
u(\Bx)=a q_\omega(\Bx) +b(\Bx),
\eeq
where $q_\omega(\Bx)$ is the solution to
\beq\label{eq:decomp02}
\begin{cases}
\Delta q_\omega + \omega^2 q_\omega = 0 \quad \mbox{in} \quad \RR^2\setminus\overline{B_1\cup B_2},\\
q_\omega(\Bx) \mbox{ satisfies the Sommerfeld radiation condition.}
\end{cases}
\eeq
The concrete form of $q_\omega$ will be shown in the next section.
Let $q_0$ be the singular function defined by
\beq\label{eq:singufun01}
q_0(\Bx):=\frac{1}{2\pi}(\ln |\Bx-\mathbf{p}_1|- \ln|\Bx-\mathbf{p}_2|),
\eeq
where $\mathbf{p}_1$ and $\mathbf{p}_2$ denote for the fixed points of the reflection $R_1R_2$ and $R_2R_1$, respectively. Here, the reflection $R_j$ with respect to $\p B_j$, centering at $\Bz_j$ and of radius $r_j$, are defined by
\beq
R_j(\Bx):=\frac{r_j^2 (\Bx-\Bz_j)}{|\Bx-\Bz_j|^2}+\Bz_j, \quad j=1, 2.
\eeq
If $\alpha_1=\alpha_2=0$, then it is proved in \cite{Yun07,Yun09} that $\mathbf{p}_j$, $j=1,2$ admits the following asymptotic expansion:
\beq
\mathbf{p}_1=\left(-\sqrt{2}\sqrt{\frac{r_{1,0}r_{2,0}}{r_{1,0}+r_{2,0}}}\sqrt{\epsilon}+\Ocal(\epsilon), 0\right)^T, \quad
\mathbf{p}_2=\left(\sqrt{2}\sqrt{\frac{r_{1,0}r_{2,0}}{r_{1,0}+r_{2,0}}}\sqrt{\epsilon}+\Ocal(\epsilon), 0\right)^T.
\eeq
In this paper, we shall consider the case that $\alpha_1, \alpha_2\neq 0$, and we derive the explicit forms of $\mathbf{p}_1$ and $\mathbf{p}_2$ as follows:
\beq\label{eq:rinv01}
\begin{split}
\mathbf{p}_1=&\left(-\frac{(r_1-r_2)\epsilon/2+\sqrt{\epsilon}\tau(r_1,r_2,\epsilon)}{r_1+r_2+\epsilon}, 0\right)^T,\\
\mathbf{p}_2=&\left(\frac{(r_2-r_1)\epsilon/2+\sqrt{\epsilon}\tau(r_1,r_2,\epsilon)}{r_1+r_2+\epsilon}, 0\right)^T,
\end{split}
\eeq
where
\beq\label{eq:deftau01}
\tau(r_1,r_2,\epsilon)=\sqrt{2r_1r_2(r_1+r_2)+(r_1^2+3r_1r_2+r_2^2)\epsilon+(r_1+r_2)\epsilon^2+\epsilon^3/4}.
\eeq
The formula \eqnref{eq:rinv01} can be verified by straightforward computations.
It can be seen that $q_0$ is the solution to the following equation (see \cite{KY09}):
\beq\label{eq:singufun02}
\begin{cases}
\Delta q_0= 0 \hspace*{1cm} \mbox{in} \quad \RR^2\setminus\overline{B_1\cup B_2},\\
q_0= C_j \hspace*{1.1cm} \mbox{on} \quad \p B_j,\\
\displaystyle{\int_{\p B_j}\p_\nu q_0|_+= (-1)^j,  \quad j=1,2,}\\
q_0(\Bx)=\Ocal(|\Bx|^{-1}) \quad \mbox{as} \quad |\Bx|\rightarrow \infty,
\end{cases}
\eeq
where $C_j$, $j=1,2,$ are
\beq\label{eq:singufun02}
C_j=(-1)^{j-1}\frac{1}{2\pi}\ln\frac{-(2r_j+\epsilon)\sqrt{\epsilon}+2\tau}{(2r_j+\epsilon)\sqrt{\epsilon}+2\tau}, \quad j=1, 2.
\eeq
In what follows, we define $b(\Bx)$ in \eqnref{eq:decomp01} by
\beq\label{eq:defb01}
b(\Bx):= u(\Bx)-\frac{\lambda_1-\lambda_2}{C_1-C_2}q_\omega(\Bx),
\eeq
where $u$, $\lambda_j$ and $C_j$, $j=1, 2$ are defined in \eqnref{eq:helm01} and \eqnref{eq:singufun02}.
We shall prove the following critical result:
\begin{lem}\label{le:01}
Suppose $\omega\cdot\epsilon^{\alpha_-}\ll 1$. Let $b(\Bx)$ be defined in \eqnref{eq:defb01}. Then for any bounded set $\Omega$ containing $\overline{B}_1$ and $\overline{B}_2$, there is a constant $C$ which is independent of $\epsilon$ and $\omega$ such that
\beq\label{eq:leessest01}
\|\nabla b\|_{L^{\infty}(\Omega\setminus\overline{B_1\cup B_2})}\leq C(1+\Ocal(\omega^2)).
\eeq
\end{lem}

\section{Quantitative approximations of the solution}
\subsection{Layer potentials}
Before the estimation of the gradient field, we introduce some necessary notations and results on the layer potential operators, which shall be need in our subsequent analysis. Let $\Gamma_\omega(\Bx)$ be the fundamental solution to PDE operator $\Delta+\omega^2$ in $\RR^2$, given by
\begin{equation}
\label{Gk} \ds \Gamma_\omega(\Bx) =
-\frac{i}{4} H_0^{(1)}(\omega|\Bx|),
 \end{equation}
 where $H_0^{(1)}(\omega|\Bx|)$ is the Hankel function of the first kind and zeroth order. We mention that if $\omega=0$ then $\Gamma_0(\Bx)=\frac{1}{2\pi}\ln|\Bx|$.
For any bounded $C^{2,\alpha}$ domain $B\subset \RR^2$, $\alpha>0$, we denote by $\Scal_B^\omega: L^2(\p B)\rightarrow H^{1}(\RR^2\setminus\p B)$ the single layer potential operator given by
\beq\label{eq:layperpt1}
\Scal_{B}^\omega[\phi](\Bx):=\int_{\p B}\Gamma_\omega(\Bx-\By)\phi(\By)\; d s_\By,
\eeq
and $(\Kcal_B^\omega)^*: L^2(\p B)\rightarrow L^2(\p B)$ the Neumann-Poincar\'e operator
\beq\label{eq:layperpt2}
(\Kcal_{B}^\omega)^{*}[\phi](\Bx):=\mbox{p.v.}\quad\int_{\p B}\frac{\p \Gamma_\omega(\Bx-\By)}{\p \nu_\Bx}\phi(\By)\; d s_\By,
\eeq
where p.v. stands for the Cauchy principle value. In \eqref{eq:layperpt2} and also in what follows, unless otherwise specified, $\nu$ signifies the exterior unit normal vector to the boundary of the concerned domain.
We also introduce the double layer potential $\Dcal_B^\omega: L^2(\p D)\rightarrow H^{1}(\RR^2\setminus\p B)$ given by
\beq\label{eq:layperpt1}
\Dcal_{B}^\omega[\phi](\Bx):=\int_{\p B}\frac{\p \Gamma_\omega(\Bx-\By)}{\p \nu_\By}\phi(\By)\; d s_\By.
\eeq
It is known that the single layer potential operator $\Scal_B^\omega$ is continuous across $\p B$ and satisfies the following trace formula
\beq \label{eq:trace}
\frac{\p}{\p\nu}\Scal_B^\omega[\phi] \Big|_{\pm} = (\pm \frac{1}{2}I+
(\Kcal_{B}^k)^*)[\phi] \quad \mbox{on} \quad \p B, \eeq
where $\frac{\p }{\p \nu}$ stands for the normal derivative and the subscripts $\pm$ indicate the limits from outside and inside of a given inclusion $B$, respectively.
The double layer potential operator $\Dcal_B^\omega$ satisfies the following trace formula across $\p B$:
\beq \label{eq:trace2}
\Dcal_B^\omega[\phi] \Big|_{\pm} = (\mp \frac{1}{2}I+
\Kcal_{B}^k)[\phi] \quad \mbox{on} \quad \p B. \eeq
When $\omega=0$ the operators $\Scal_B^0$ and $\Dcal_B^0$ stand for the single layer potential operator and double layer potential with kernel function $\Gamma_0$.

\subsection{Asymptotic estimates}
Recall that the Bessel function $J_0(\omega|\Bx)$ and the Neumann function $Y_0(\omega|\Bx)$ admit the following integral formula (see, e.g., \cite{MiIr72}):
\beq\label{eq:intforHan01}
\begin{split}
J_0(\omega|\Bx|)&=\frac{1}{\pi}\int_0^{\pi} \cos(\omega|\Bx|\cos\theta)d\theta, \\
Y_0(\omega|\Bx|)&=\frac{4}{\pi^2}\int_0^{\pi/2}\cos(\omega|\Bx|\cos\theta)(\gamma+\ln(2\omega|\Bx|\sin^2\theta))d\theta,
\end{split}
\eeq
where $\gamma=0.5772...$ is the Euler-Mascheroni constant.
The Hankel function appeared in \eqnref{Gk} can be represented by
\beq\label{eq:intforHan02}
H_0^{(1)}(\omega|\Bx|)=-\frac{i}{4}J_0(\omega|\Bx|)+\frac{1}{4}Y_0(\omega|\Bx|).
\eeq
Note that for $\omega$ sufficiently small, one has the following asymptotic result:
\beq\label{eq:asyHan01}
\Gamma_\omega(\Bx)=a_\omega+\Gamma_0(\Bx)+A_\omega(\Bx),
\eeq
where $a_\omega$ is a constant defined by
$$
a_\omega:=-\frac{i}{4}+\frac{\gamma}{2\pi}+\frac{1}{2\pi}\ln\frac{\omega}{2},
$$
and the function $A_\omega(\Bx)$ is defined by
\beq\label{eq:asyHan02}
\begin{split}
A_\omega(\Bx):=&\frac{i}{4\pi}|\Bx|\int_0^{\pi}\sin(\eta|\Bx|\cos\theta)\cos\theta d\theta\\
&-\frac{1}{\pi^2}|\Bx|\int_0^{\pi/2}\sin(\eta|\Bx|\cos\theta)\cos\theta (\gamma+\ln(2\omega|\Bx|\sin^2\theta)) d\theta,
\end{split}
\eeq
where $\eta\in(0, \omega)$ is some fixed positive number. It is worth mentioning that $A_\omega$ is a smooth function in $\RR^2$ for any $\omega\in \RR_+$. Besides one has
\beq\label{eq:asympak01}
A_\omega(\Bx)=-\frac{1}{4\pi}|\Bx|^2\omega^2\ln \omega+\Ocal(\omega^2).
\eeq
We define the boundary integral $\mathcal{A}_B^\omega$ by
\beq\label{eq:defAk01}
\mathcal{A}_B^\omega[\phi](\Bx):=\int_{\p B} A_\omega(\Bx-\By)\phi(\By) ds_\By.
\eeq
In the sequel, we let $q_\omega$ be the following singular function:
\beq\label{eq:defqk01}
q_\omega:=\Gamma_\omega(\Bx-\mathbf{p}_1)-\Gamma_\omega(\Bx-\mathbf{p}_2)=q_0+A_\omega(\Bx-\mathbf{p}_1)-A_\omega(\Bx-\mathbf{p}_2).
\eeq
\subsection{First approximation}\label{sec:fst01}
We next consider the solution to \eqnref{eq:helm00}. By imploring the layer potential techniques, one can represent the solution to \eqnref{eq:helm00} by
\beq
u^*=\left\{
\begin{split}
&u^i+\Scal_{B_c}^{\omega}[\varphi_1^*]\hspace*{.5cm}  \mbox{in} \quad \RR^2\setminus\overline{B_c},\\
&\Scal_{B_c}^{k_c}[\varphi_2^*]\hspace*{1.38cm}  \mbox{in} \quad B_c,
\end{split}
\right.
\eeq
where $B_c:=B_1\cup B_2$ and $k_c=\omega\sqrt{\varepsilon_1}$. By using the transmission conditions across $\p B_c$, there holds
\beq\label{eq:transint01}
\begin{split}
\mathbf{A}_{B_c}^{\omega}[\bm{\varphi}^*]=\bm{U}   \quad \mbox{on} \quad \p B_c,
\end{split}
\eeq
where the operator $\mathbf{A}_{B_c}^{\omega}: H^{-1/2}(\p B_c)\times H^{-1/2}(\p B_c)\rightarrow H^{1/2}(\p B_c)\times H^{-1/2}(\p B_c)$ is defined by
\beq\label{eq:transint02}
\mathbf{A}_{B_c}^{\omega}:=\left(
\begin{array}{cc}
-\Scal_{B_c}^{\omega} & \Scal_{B_c}^{k_c} \\
-\big(\frac{I}{2}+(\Kcal_{B_c}^{\omega})^*\big) & \frac{1}{\varepsilon_1} \big(-\frac{I}{2}+(\Kcal_{B_c}^{k_c})^*\big)
\end{array}
\right),
\eeq
and
\beq\label{eq:transint03}
\bm{\varphi}^*=\left(
\begin{array}{c}
\varphi_1^* \\
\varphi_2^*
\end{array}
\right), \quad
\bm{U}=\left(
\begin{array}{c}
u^i \\
\frac{\p u^i}{\p \nu}
\end{array}
\right).
\eeq
For the later use, we define the operator $\mathbb{S}$ by
\beq\label{eq:defSk01}
\mathbb{S}:=
\left(
\begin{array}{cc}
\Scal_{B_1}^0|_{\p B_1}  & \Scal_{B_2}^0|_{\p B_1} \\
\Scal_{B_1}^0|_{\p B_2}  & \Scal_{B_2}^0|_{\p B_2}
\end{array}
\right),
\eeq
and the operator $\mathbb{K}^*$ by
\beq
\mathbb{K}^*:=\left(
    \begin{array}{cc}
       (\Kcal_{B_1}^{0})^* &  \p_{\nu_1}\Scal_{B_2}^0 \medskip \\
     \p_{\nu_2}\Scal_{B_1}^0 & (\Kcal_{B_2}^{0})^*\\
    \end{array}
  \right),
\eeq
where $\nu_1$ and $\nu_2$ are the unit normal directions to $\p B_1$ and $\p B_2$, respectively.
It can be verified that $\mathbb{S}=\Scal_{B_c}^0$ and $\mathbb{K}^*=(\Kcal_{B_c}^0)^*$.
Similar to the Calder\'on type identity introduced in \cite{ACKL13}, we have the following identity:
\beq\label{eq:cald01}
\mathbb{S}\mathbb{K}^*=\mathbb{K}\mathbb{S},
\eeq
where $\mathbb{K}$ is the ajoint operator of $\mathbb{K}^*$ given by
$$
\mathbb{K}:=\left(
    \begin{array}{cc}
       \Kcal_{B_1}^{0} &  \Dcal_{B_2}^0|_{\p B_1} \medskip \\
     \Dcal_{B_1}^0|_{\p B_2} & \Kcal_{B_2}^{0}\\
    \end{array}
  \right).
$$
For completeness and convenient reference to the reader, we shall present the proof to the identity \eqnref{eq:cald01} in Appendix \ref{app:cald01}.

\begin{proof}[Proof of Lemma \ref{le:decom01}]
By using the asymptotic estimates in the previous section, one can derive the following asymptotic expansions for the layer potentials:
\beq\label{eq:asymsgdb01}
\begin{split}
\Scal_{B_c}^{\omega}[\varphi]=&a_\omega\int_{\p B_c} \varphi + \Scal_{B_c}^0[\varphi]+\mathcal{A}_{B_c}^\omega[\varphi],\\
(\Kcal_{B_c}^{\omega})^*[\varphi]=&(\Kcal_{B_c}^{0})^*[\varphi]+\p_\nu\mathcal{A}_{B_c}^\omega[\varphi].
\end{split}
\eeq
By using \eqnref{eq:transint01} and the definition of $a_\omega$, one has
\beq\label{eq:phiint01}
\int_{\p B_c}\varphi_1^*=\int_{\p B_c}\varphi_2^* +\Ocal(\omega).
\eeq
We declare that there holds the decomposition $u^*=u+u'$ in $\RR^2\setminus\overline{B_1\cup B_2}$, where $u'$ is the solution to
\beq\label{eq:helm02}
\begin{cases}
\Delta u' + \omega^2 u' = 0, \quad \mbox{in} \quad \RR^2\setminus\overline{B_1\cup B_2}\\
u'=\Ocal(\omega) , \quad \mbox{on} \quad \p B_1\cup \p B_2\\
u'(\Bx) \mbox{ satisfies the Sommerfeld radiation condition,}
\end{cases}
\eeq
together with the relation
\beq
\int_{\p B_j} \p_\nu u'=\Ocal(\omega^2), \quad j=1, 2.
\eeq
In fact, we assume that $\varphi_1^*$ and $\varphi_2^*$ admits the following asymptotic expansions:
\[
\begin{split}
\varphi_1^*=&\varphi_{1,0}^*+\omega\ln\omega\varphi_{1,1}^*+\omega\varphi_{1,2}^*+\omega^2\ln \omega\varphi_{1,3}^*+\Ocal(\omega^2), \\
\varphi_2^*=&\varphi_{2,0}^*+\omega\ln\omega\varphi_{2,1}^*+\omega\varphi_{2,2}^*+\omega^2\ln \omega\varphi_{2,3}^*+\Ocal(\omega^2).
\end{split}
\]
It then follows from \eqnref{eq:transint01} and the asymptotic expansion \eqnref{eq:asymsgdb01} that
\beq
\left(-\frac{I}{2}+\mathbb{K}^*\right)[\varphi_{2,0}^*+\omega\ln \omega\varphi_{2,1}^*]=0, \quad \int_{\p B_j} \varphi_{1}^*=\Ocal(\omega^2).
\eeq
Thus one has
\beq
\mathbb{S}[\varphi_{2,0}^*]=\lambda_{1,1}\chi(\p B_1)+\lambda_{2,1}\chi(\p B_2), \quad \mathbb{S}[\varphi_{2,1}^*]=\lambda_{1,2}\chi(\p B_1)+\lambda_{2,2}\chi(\p B_2),
\eeq
where $\lambda_{j,l}$, $j, l=1, 2$ are constants. It follows by straightforward computations that
\beq
\begin{split}
\Scal_{B_c}^{0}[\varphi_{2,2}^*]=&-2\frac{\epsilon_1}{\omega}\left(\Scal_{B_c}^{0}[\varphi_{2,0}^*]+\Scal_{B_c}^{0}[\p_\nu u_0^i]-\left(-\frac{I}{2}+\Kcal_{B_c}^{0}\right)[u_0^i]\right),\\
\Scal_{B_c}^{0}[\varphi_{2,3}^*]=&-2\frac{\epsilon_1}{\omega}\Scal_{B_c}^{0}[\varphi_{2,1}^*].
\end{split}
\eeq
One thus has
\beq
\begin{split}
\Scal_{B_c}^{k_c}[\varphi_2^*]=&\Scal_{B_c}^{0}[\varphi_{2,0}^*+\omega\ln \omega\varphi_{2,1}^*+\omega^2\ln \omega\varphi_{2,3}^*]+\omega \Scal_{B_c}^{0}[\varphi_{2,2}^*]+\Ocal(\omega^2)\\
=&\left\{
\begin{array}{ll}
\lambda_1-2\epsilon_1 \Scal_{B_c}^{0}[\p_\nu u_0^i]-\epsilon_1 u_0^i+\Ocal(\omega^2) \quad \mbox{on} \quad \p B_1,\medskip\\
\lambda_2-2\epsilon_1 \Scal_{B_c}^{0}[\p_\nu u_0^i]-\epsilon_1 u_0^i+\Ocal(\omega^2) \quad \mbox{on} \quad \p B_2.
\end{array}
\right.
\end{split}
\eeq
One can thus set
$$
u=u^i+\Scal_{B_c}^{\omega}[\varphi_1^*]-2\epsilon_1 \Scal_{B_c}^{0}[\p_\nu u_0^i]-\epsilon_1 u_0^i +\Ocal(\omega^2),
$$
and the higher order term is arranged such that $\Delta u+ \omega^2 u=0$ holds in $\RR^2\setminus\overline{B_1\cup B_2}$.
Now it is readily verified that $u'=u^*-u$ satisfies \eqnref{eq:helm02}. More precisely, one has
$$
u'=-2\epsilon_1 \Scal_{B_c}^{0}[\p_\nu u_0^i]-\epsilon_1 u_0^i+\Ocal(\omega^2) \quad \mbox{on} \quad \p B_1\cup \p B_2.
$$
Suppose $u'=\epsilon_1 u_1' +\Ocal(\omega^2)$, where $u_1'$ is the solution to
\beq\label{eq:helm03}
\begin{cases}
\Delta u_1' = 0 \hspace*{3cm} \mbox{in} \quad \RR^2\setminus\overline{B_1\cup B_2},\smallskip\\
\displaystyle{u_1'=-2\Scal_{B_c}^{0}[\p_\nu u_0^i]-u_0^i} \hspace*{.4cm} \mbox{on} \quad \p B_1\cup \p B_2,\smallskip\\
u_1'(\Bx)=\Ocal(|\Bx|^{-1}).
\end{cases}
\eeq

We mention that $\nabla u_1'$ is uniformly bounded with respect to the distance $\epsilon$. In fact, the solution to \eqnref{eq:helm03} can
be represented by
$$
u_1' =\Scal_{B_c}[\varphi'](\Bx), \quad \Bx\in \RR^2\setminus\overline{B_1\cup B_2},
$$
where the $\varphi'$ satisfy
$$
\int_{\p B_c}\varphi'=0,
$$
and
\beq\label{eq:intbyp01}
\left(-\frac{I}{2} +\mathbb{K}^*\right)[\varphi']=-2\mathbb{K}^*[\p_\nu u_0^i] \quad \mbox{on} \quad \p B_c.
\eeq
One can show that
$$
u_1'(\zeta_1)-u_1'(-\zeta_1)=\epsilon(\p_\nu u_0^i(\zeta_1)-2\p_\nu u_0^i(-\zeta_1))+\Ocal(\epsilon^2),
$$
where $\zeta_1=(\frac{\epsilon}{2}, 0)$. One can then prove that $u_1$ is uniformly bounded by using the same strategy in the proof of Lemma \ref{le:01}.
\end{proof}
\subsection{Further approximation}
In order to prove the main result, we need to estimate the key quantities at the right hand side of \eqnref{eq:decomp01}, where
$$
a=\frac{\lambda_1-\lambda_2}{C_1-C_2}.
$$
By using \eqnref{eq:asympak01}, one has
\beq\label{eq:estlam01}
\begin{split}
\lambda_2-\lambda_1=&\int_{\p B_2} u\p_\nu q_0+\int_{\p B_1} u\p_\nu q_0 +\Ocal(\omega^2)\\
=&\int_{\p B_2} (u-u^i)\p_\nu q_\omega+\int_{\p B_1} (u-u^i)\p_\nu q_\omega +\int_{\p B_1\cup \p B_2} u^i \p_\nu q_0 \\
&+\frac{1}{2\pi} \omega^2\ln \omega\int_{B_1\cup B_2} \nabla u^i \cdot (\mathbf{p}_1-\mathbf{p}_2)+\Ocal(\omega^2)\\
=&\int_{\p B_1\cup \p B_2}\p_\nu (u-u^i) q_\omega+\int_{\p B_1\cup \p B_2} u^i \p_\nu q_0+\Ocal(\omega^2)\\
=&\frac{1}{\pi}\omega^2\ln \omega\int_{B_1\cup B_2} \nabla u^i \cdot (\mathbf{p}_1-\mathbf{p}_2)+\int_{\p B_1\cup \p B_2} u^i \p_\nu q_0+\Ocal(\omega^2)\\
=&u^i(\mathbf{p}_1)-u^i(\mathbf{p}_2)+\frac{1}{\pi}\omega^2\ln \omega\int_{B_1\cup B_2} \nabla u^i \cdot (\mathbf{p}_1-\mathbf{p}_2)+\Ocal(\omega^2),
\end{split}
\eeq
where we have used the results
\[
\begin{split}
&\int_{\p B_1\cup \p B_2} (u-u^i) \p_\nu(q_\omega-q_0)\\
=& \frac{1}{4\pi} \omega^2\ln \omega\int_{\p B_1\cup \p B_2} (u-u^i) \p_\nu (|\Bx-\mathbf{p}_2|^2-|\Bx-\mathbf{p}_1|^2)\\
=&\frac{1}{2\pi} \omega^2\ln \omega\int_{\p B_1\cup \p B_2} (u-u^i) \nu \cdot (\mathbf{p}_1-\mathbf{p}_2)\\
=&\frac{1}{2\pi} \omega^2\ln \omega\int_{B_1\cup B_2} \nabla u^i \cdot (\mathbf{p}_1-\mathbf{p}_2),
\end{split}
\]
and
\[
\begin{split}
&\int_{\p B_1\cup \p B_2}\p_\nu (u-u^i) (q_\omega-q_0)\\
=&\frac{1}{2\pi}\omega^2\ln \omega\left(r_1^2\int_{\p B_1}\p_\nu (u-u^i)\frac{\Bx-\Bz_1}{|\Bx-\Bz_1|^2}+r_2^2\int_{\p B_2}\p_\nu (u-u^i)\frac{\Bx-\Bz_2}{|\Bx-\Bz_2|^2}\right)\cdot(\mathbf{p}_1-\mathbf{p}_2)\\
&+\Ocal(\omega^2)=\frac{1}{2\pi} \omega^2\ln \omega\int_{B_1\cup B_2} \nabla u^i \cdot (\mathbf{p}_1-\mathbf{p}_2)+\Ocal(\omega^2).
\end{split}
\]
Moreover, one has
\beq\label{eq:asymnqk01}
\begin{split}
\nabla q_\omega=&\nabla q_0 +\omega^2\ln \omega\frac{1}{2\pi}(\mathbf{p}_1-\mathbf{p}_2)+\Ocal(\omega^2)\\ =&\frac{1}{2\pi}\left(\frac{\Bx-\mathbf{p}_1}{|\Bx-\mathbf{p}_1|^2}-\frac{\Bx-\mathbf{p}_2}{|\Bx-\mathbf{p}_2|^2}\right)+\omega^2\ln \omega\frac{1}{2\pi}(\mathbf{p}_1-\mathbf{p}_2)+\Ocal(\omega^2). \\
\end{split}
\eeq

\section{Estimate of $b(\Bx)$}
By definition of \eqnref{eq:defb01}, one finds that $b(\Bx)$ is the solution to
\beq\label{eq:bxeq01}
\begin{cases}
\Delta b + \omega^2 b = 0 \hspace*{3.4cm} \mbox{in} \quad \RR^2\setminus\overline{B_1\cup B_2},\smallskip\\
b= \big(\lambda_2C_1-\lambda_1C_2\big)/(C_1-C_2) \quad \mbox{on} \quad \p B_1\cup \p B_2,\smallskip\\
(b-u^i)(\Bx) \mbox{ satisfies the Sommerfeld radiation condition.}
\end{cases}
\eeq
By using layer potential techniques, one can represent $b$ in \eqnref{eq:bxeq01} by
\beq\label{eq:bxeq05}
b(\Bx)=u^i(\Bx)+\Scal_{B_1}^\omega[\varphi_1](\Bx)+\Scal_{B_2}^\omega[\varphi_2](\Bx),
\eeq
where $\varphi_1\in L^2(\p B_1)$ and $\varphi_2\in L^2(\p B_2)$ satisfy
\beq\label{eq:bxeq06}
\begin{split}
u^i(\Bx)+\Scal_{B_1}^\omega[\varphi_1](\Bx)+\Scal_{B_2}^\omega[\varphi_2](\Bx)=\tilde{C}_1, \quad \Bx\in \p B_1\cup \p B_2,
\end{split}
\eeq
with $\tilde{C}_1:=\big(\lambda_2C_1-\lambda_1C_2\big)/(C_1-C_2)$.

Note that it is proved in \cite{KLY13} that $\nabla b(\Bx)$ is uniformly bounded if $\omega=0$.
We need some further analysis on the solution $b$. First, by using \eqnref{eq:bxeq06} and the expansion \eqnref{eq:intforHan02} one has
\beq
a_\omega\int_{\p B_1\cup\p B_2}\bm\varphi+\mathbb{S}[\bm\varphi]+\mathbb{A}^\omega[\bm\varphi]=\tilde{C}_1-u^i \quad \mbox{on} \quad \p B_1\cup \p B_2,
\eeq
for $\omega$ sufficiently small. Here $\bm\varphi=(\varphi_1,\varphi_2)$ and the operator $\mathbb{S}$ is given by \eqnref{eq:defSk01}. The operator $\mathbb{A}^\omega$ is given by
\beq
\mathbb{A}^\omega:=
\left(
\begin{array}{cc}
\Acal_{B_1}^\omega|_{\p B_1}  & \Acal_{B_2}^\omega|_{\p B_1} \\
\Acal_{B_1}^\omega|_{\p B_2}  & \Acal_{B_2}^\omega|_{\p B_2}
\end{array}
\right).
\eeq
By using the definition of $a_\omega$ there holds:
\beq
\int_{\p B_1\cup\p B_2}\bm\varphi=\Ocal(\omega).
\eeq
Suppose $\varphi_j=\varphi_{j,0} + \Ocal(\omega)$, $j=1, 2$.
Direct asymptotic analysis shows that
\beq\label{eq:asyb01}
b(\Bx)= u^i(\Bx) + b_0(\Bx)+\Ocal(\omega^2),
\eeq
where $b_0=\Scal_{B_1}^0[\varphi_{1,0}](\Bx)+\Scal_{B_2}^0[\varphi_{2,0}](\Bx)$ is the harmonic function which satisfies
\beq\label{eq:b0eq01}
\begin{cases}
\Delta b_0 = 0 \hspace*{1.2cm} \mbox{in} \quad \RR^2\setminus\overline{B_1\cup B_2},\smallskip\\
b_0= \tilde{C}_1-u^i\hspace*{.4cm} \mbox{on} \quad \p B_1\cup \p B_2,\smallskip\\
b_0(\Bx)=\Ocal(|\Bx|^{-1}).
\end{cases}
\eeq

\begin{proof}[Proof of Lemma \ref{le:01}] By the asymptotic result in \eqnref{eq:asyb01}, it is sufficient to prove that $\nabla b_0$, where $b_0$ is the solution to \eqnref{eq:b0eq01}, is uniformly bounded in $\RR^2\setminus (B_1\cup B_2)$. Since $\nabla b_0$ is harmonic in $\RR^2\setminus \overline{B_1\cup B_2}$, and $\nabla b_0=\Ocal(|\Bx|^{-2})$, the function $|\nabla b_0|_{l^\infty}$ in $\RR^2\setminus (B_1\cup B_2)$ attains its maximum on the boundary $\p B_1\cup \p B_2$.
Note that $b_0$ is smooth on $\RR^2\setminus\overline{B_1\cup B_2}$. Tt is enough to show that $\nabla b_0$ is uniformly bounded with respect to $\epsilon$ on the two points $\zeta_1$ and $-\zeta_1$, where $\zeta_1=(\frac{\epsilon}{2}, 0)$.
Since $b_0(\zeta_1)=\tilde{C}_1-u^i(\zeta_1)$, one has
\beq
\begin{split}
\nabla b_0(\zeta_1)=&\nu(\zeta_1)\cdot\nabla b_0(\zeta_1) \nu(\zeta_1) + \p_T b_0(\zeta_1)T(\zeta_1)\\
=&\p_{\Bx_1} b_0(\zeta_1) (-1, 0)+\p_T u^i(\zeta_1)T(\zeta_1)\\
=& (-1, 0)\lim_{\epsilon\rightarrow 0}\frac{b_0(-\zeta_1)-b_0(\zeta_1)}{\epsilon}+(0,1)\p_T u^i(\zeta_1)\\
=& \nabla u^i(\zeta_1),
\end{split}
\eeq
where $\p_T$ stands for the tangential derivative and $T$ is the unit tangential vector. Since $\nabla u^i(\zeta_1)$ is uniformly bounded, one thus has verified that $\nabla b_0(\zeta_1)$ is uniformly bounded with respect to $\epsilon$. Similarly, one can show that $\nabla b_0(-\zeta_1)$ is also uniformly bounded.
\end{proof}
\begin{rem}
We mention that the bound on $\nabla b_0$ can be shown by following a similar argument in \cite{BLY09} and \cite{KLY13}. Here, we provide a different proof.
\end{rem}

\section{Estimate of $q_\omega(\Bx)$}
In this section, we shall estimate the singular function $q_\omega(\Bx)$. The asymptotic result \eqnref{eq:asymnqk01} shows that one only needs to estimate $\nabla q_0$. We mention that if $r_1$ and $r_2$ are constants which do not depend on $\epsilon$, the estimate of $\nabla q_0$ is well settled in \cite{KLY13}. We shall consider the case that $r_1$ and $r_2$ depend on $\epsilon$. Note that
$$
|\Bx-\mathbf{p}_1|\geq \frac{-(2r_2+\epsilon)\epsilon+\sqrt{\epsilon}\tau}{r_1+r_2+\epsilon}, \quad \mbox{and} \quad
|\Bx-\mathbf{p}_2|\geq \frac{-(2r_1+\epsilon)\epsilon+\sqrt{\epsilon}\tau}{r_1+r_2+\epsilon},
$$
hold for $\Bx\in \RR^2\setminus\overline{B_1\cup B_2}$. It follows that
\beq
\begin{split}
&\left\|\frac{\Bx-\mathbf{p}_1}{|\Bx-\mathbf{p}_1|^2}-\frac{\Bx-\mathbf{p}_2}{|\Bx-\mathbf{p}_2|^2}\right\|_{L^{\infty}(\RR^2\setminus\overline{B_1\cup B_2})}\\
\leq&\left\|\frac{\Bx-\mathbf{p}_1}{|\Bx-\mathbf{p}_1|^2}\right\|_{L^{\infty}(\RR^2\setminus\overline{B_1\cup B_2})}
+\left\|\frac{\Bx-\mathbf{p}_2}{|\Bx-\mathbf{p}_2|^2}\right\|_{L^{\infty}(\RR^2\setminus\overline{B_1\cup B_2})}\\
\leq&\Big(\frac{1}{-(2r_1+\epsilon)\epsilon/2+\sqrt{\epsilon}\tau}+\frac{1}{-(2r_2+\epsilon)\epsilon/2+\sqrt{\epsilon}\tau}\Big)(r_1+r_2+\epsilon).
\end{split}
\eeq
On the other hand, setting $\Bx=(\frac{\epsilon}{2},0)^T$, one has
\[
\begin{split}
&\left|\frac{\Bx-\mathbf{p}_1}{|\Bx-\mathbf{p}_1|^2}-\frac{\Bx-\mathbf{p}_2}{|\Bx-\mathbf{p}_2|^2}\right|_{l^\infty}\\
=&\Big(\frac{1}{-(2r_1+\epsilon)\epsilon/2+\sqrt{\epsilon}\tau}+\frac{1}{(2r_1+\epsilon)\epsilon/2+\sqrt{\epsilon}\tau}\Big)(r_1+r_2+\epsilon).
\end{split}
\]
Similarly, setting $\Bx=(-\frac{\epsilon}{2},0)^T$, one has
\[
\begin{split}
&\left|\frac{\Bx-\mathbf{p}_1}{|\Bx-\mathbf{p}_1|^2}-\frac{\Bx-\mathbf{p}_2}{|\Bx-\mathbf{p}_2|^2}\right|_{l^\infty}\\
=&\Big(\frac{1}{-(2r_2+\epsilon)\epsilon/2+\sqrt{\epsilon}\tau}+\frac{1}{(2r_2+\epsilon)\epsilon/2+\sqrt{\epsilon}\tau}\Big)(r_1+r_2+\epsilon).
\end{split}
\]
Thus there holds
\beq\label{eq:estinaq01}
\begin{split}
&\frac{r_1+r_2+\epsilon}{-(\max(r_1,r_2)+\epsilon/2)\epsilon+\sqrt{\epsilon}\tau}\leq \left\|\nabla q_0\right\|_{L^{\infty}(\RR^2\setminus\overline{B_1\cup B_2})}\leq 2\frac{r_1+r_2+\epsilon}{-(\max(r_1,r_2)+\epsilon/2)\epsilon+\sqrt{\epsilon}\tau}.
\end{split}
\eeq
\begin{proof}[Proof of Theorem \ref{th:main01}]
It is sufficient to show the estimation for $\nabla u$ in $\Omega\setminus\overline{B_1\cup B_2}$.
By using \eqnref{eq:singufun02} one has
\beq
C_1-C_2=\frac{1}{2\pi}\left(\ln\left(1-\frac{2(r_1+\epsilon/2)\sqrt{\epsilon}}{(r_1+\epsilon/2)\sqrt{\epsilon}+\tau}\right)
+\ln\left(1-\frac{2(r_2+\epsilon/2)\sqrt{\epsilon}}{(r_2+\epsilon/2)\sqrt{\epsilon}+\tau}\right)\right).
\eeq
Firstly, if $\alpha_-< 1$, then one has
\beq\label{eq:esttau01}
\tau=C_0\epsilon^{\alpha_-+\min(\alpha_+,1)/2}(1+o(1)),
\eeq
where $C_0>0$ does not depend on $\epsilon$ and is a generic constant which may vary for different choice of $\alpha_1$ and $\alpha_2$. It follows that
\[
C_1-C_2=\left\{
\begin{array}{ll}
-2\frac{r_-}{C_0}\epsilon^{1/2-\alpha_+/2}(1+o(1)), & \alpha_+<1, \\
\ln \left(1-\frac{2r_-}{r_-+C_0}\right), & \alpha_+\geq 1,
\end{array}
\right.
\]
where $\alpha_-$ is defined in \eqnref{eq:thdenr01}.
By \eqnref{eq:estinaq01}, one has
\[
\left\|\nabla q_0\right\|_{L^{\infty}(\RR^2\setminus\overline{B_1\cup B_2})}\sim
\left\{
\begin{array}{ll}
\frac{r_-}{C_0}\epsilon^{-\frac{1}{2}-\frac{\alpha_+}{2}}, & \alpha_+<1 \\
\frac{r_-}{C_0}\epsilon^{-1}. & \alpha_+\geq 1
\end{array}
\right.
\]
Finally by \eqnref{eq:estlam01}, one can derive that
\[
\begin{split}
\lambda_1-\lambda_2=&\nabla u^i(\mathbf{p}_2)\cdot (\mathbf{p}_2-\mathbf{p}_1)+\frac{1}{\pi}\omega^2\ln \omega\int_{B_1\cup B_2} \nabla u^i \cdot (\mathbf{p}_2-\mathbf{p}_1)+\Ocal(\omega^2+|\mathbf{p}_1-\mathbf{p}_2|^2)\\
=&2\frac{1}{r_-}\epsilon^{1/2+\min(\alpha_+,1)/2}(1+o(1))\left(\p_{\Bx_1} u^i(\mathbf{0})+\frac{1}{\pi}\omega^2\ln \omega\int_{B_1\cup B_2} \p_{\Bx_1} u^i +\Ocal(\omega^2)\right),
\end{split}
\]
which together with Lemma \ref{le:01} further yields that
\beq
\begin{split}
\|\nabla u\|_{L^\infty}\sim &\Big|\frac{\lambda_1-\lambda_2}{C_1-C_2}\Big| \|\nabla q_\omega\|_{L^\infty}+\Ocal(\omega^2)\\
\sim &\frac{C_0}{r_-}\epsilon^{\min(\alpha_+,1)/2-1/2}\left(\p_{\Bx_1} u^i(\mathbf{0})+\frac{1}{\pi}\omega^2|\ln \omega|\int_{B_1\cup B_2} \p_{\Bx_1} u^i +\Ocal(\omega^2)\right)+\Ocal(1).
\end{split}
\eeq
In the above estimation, $L^\infty$ stands for $L^\infty(\Omega\setminus\overline{B_1\cup B_2})$.
Similarly, if $\alpha_-\geq 1$ then one can derive that
$$
\tau=C_0\epsilon^{3/2}(1+o(1)).
$$
It then follows the estimates
$C_1-C_2=\mathcal{O}(1)$, $\left\|\nabla q_0\right\|_{L^{\infty}(\RR^2\setminus\overline{B_1\cup B_2})}=O(\epsilon^{-1})$ and $\mathbf{p}_2-\mathbf{p}_1=O(\epsilon)$ and thus
$\|\nabla u\|_{L^\infty(\Omega\setminus\overline{B_1\cup B_2})}$ is uniformly bounded. The proof is complete.
\end{proof}

\section*{Acknowledgements}
The work of Y. Deng was supported by NSF grant of China No. 11971487 and NSF grant of Hunan No. 2020JJ2038. The work of X. Fang was supported by NSF Basic Science Center Project No. 72088101, NSF grant China No. 72001077, Humanities and Social Sciences Foundation of the Ministry of Education No. 20YJC910005.
The work of H. Liu is supported by a startup fund from City University of Hong Kong and the Hong Kong RGC General Research Funds (projects 12301420, 12302919 and 12301218).

\appendix
\section{Calder\'on type identity} \label{app:cald01}
In this appendix, we prove the Calder\'on type identity \eqnref{eq:cald01}. By straightforward computations one can show that
\[
\mathbb{S}\mathbb{K}^*=
\left(
\begin{array}{cc}
\Scal_{B_1}^{0}(\Kcal_{B_1}^{0})^*|_{\p B_1}+\Scal_{B_2}^0\p_{\nu_2} \Scal_{B_1}^0|_{\p B_1}  & \Scal_{B_1}^0\p_{\nu_1} \Scal_{B_2}^0|_{\p B_1} +\Scal_{B_2}^0(\Kcal_{B_2}^0)^*|_{\p B_1} \\
\Scal_{B_1}^{0}(\Kcal_{B_1}^{0})^*|_{\p B_2}+\Scal_{B_2}^0\p_{\nu_2} \Scal_{B_1}^0|_{\p B_2}  &  \Scal_{B_1}^0\p_{\nu_1} \Scal_{B_2}^0|_{\p B_2} +\Scal_{B_2}^0(\Kcal_{B_2}^0)^*|_{\p B_2}
\end{array}
\right),
\]
and
\[
\mathbb{S}\mathbb{K}^*=
\left(
\begin{array}{cc}
\Kcal_{B_1}^{0}\Scal_{B_1}^{0}+\Dcal_{B_2}^0\Scal_{B_1}^0|_{\p B_1}   & \Kcal_{B_1}^{0}\Scal_{B_2}^{0}+\Dcal_{B_2}^0\Scal_{B_2}^0|_{\p B_1}  \\
\Dcal_{B_1}^{0}\Scal_{B_1}^{0}|_{\p B_2}+\Kcal_{B_2}^0\Scal_{B_1}^0   & \Dcal_{B_1}^{0}\Scal_{B_2}^{0}|_{\p B_2}+\Kcal_{B_2}^0\Scal_{B_2}^0
\end{array}
\right).
\]
Note that there holds the Caldr\'on identity:
$$\Scal_{B_1}^{0}(\Kcal_{B_1}^{0})^*|_{\p B_1}=\Kcal_{B_1}^{0}\Scal_{B_1}^{0}, \quad \Scal_{B_2}^{0}(\Kcal_{B_2}^{0})^*|_{\p B_2}=\Kcal_{B_2}^{0}\Scal_{B_2}^{0}.$$
We first show the identity
$$\Scal_{B_2}^0\p_{\nu_2} \Scal_{B_1}^0|_{\p B_1}=\Dcal_{B_2}^0\Scal_{B_1}^0|_{\p B_1}.$$
In fact, letting $\varphi\in L^2(\p B_1)$ and by integration by parts, there holds
\[
\begin{split}
\Scal_{B_2}^0\p_{\nu_2} \Scal_{B_1}^0[\varphi](\Bx)&=\int_{\p B_2}\Gamma_0(\Bx-\By)\int_{\p B_1}\frac{\p \Gamma_0(\By-\Bz)}{\p_{\nu_\By}}\varphi(\Bz)d s_\Bz d s_\By\\
&=\int_{\p B_1}\int_{\p B_2}\Gamma_0(\Bx-\By)\frac{\p \Gamma_0(\By-\Bz)}{\p_{\nu_\By}} d s_\By\varphi(\Bz)d s_\Bz \\
&=\int_{\p B_1}\int_{\p B_2}\frac{\p \Gamma_0(\Bx-\By)}{\p_{\nu_\By}}\Gamma_0(\By-\Bz) d s_\By\varphi(\Bz)d s_\Bz \\
&=\int_{\p B_2}\frac{\p \Gamma_0(\Bx-\By)}{\p_{\nu_\By}}\int_{\p B_1}\Gamma_0(\By-\Bz)\varphi(\Bz)d s_\Bz d s_\By=\Dcal_{B_2}^0\Scal_{B_1}^0[\varphi](\Bx),
\end{split}
\]
for any $\Bx\in \p B_1$.
Next, by integration by parts again one has
\[
\begin{split}
&\Scal_{B_1}^0\p_{\nu_1} \Scal_{B_2}^0[\varphi](\Bx) +\Scal_{B_2}^0(\Kcal_{B_2}^0)^*[\varphi](\Bx)\\
=&\int_{\p B_1}\Gamma_0(\Bx-\By)\int_{\p B_2}\frac{\p \Gamma_0(\By-\Bz)}{\p_{\nu_\By}}\varphi(\Bz)d s_\Bz d s_\By\\
&+\int_{\p B_2}\Gamma_0(\Bx-\By)\int_{\p B_2}\frac{\p \Gamma_0(\By-\Bz)}{\p_{\nu_\By}}\varphi(\Bz)d s_\Bz \Big|_- d s_\By+
\frac{1}{2}\int_{\p B_2}\Gamma_0(\Bx-\By)\varphi(\By)\Big|_- d s_\By\\
=&\int_{\p B_2}\int_{\p B_1}\frac{\p \Gamma_0(\Bx-\By)}{\p_{\nu_\By}}\Big|_- \Gamma_0(\By-\Bz)d s_\By\varphi(\Bz)d s_\Bz-
\frac{1}{2}\int_{\p B_2}\Gamma_0(\Bx-\By)\varphi(\By)\Big|_- d s_\By\\
&+\int_{\p B_2}\int_{\p B_2}\frac{\p \Gamma_0(\Bx-\By)}{\p_{\nu_\By}} \Gamma_0(\By-\Bz)d s_\By\varphi(\Bz)d s_\Bz\\
=&\Kcal_{B_1}^{0}\Scal_{B_2}^{0}[\varphi](\Bx)+\Dcal_{B_2}^0\Scal_{B_2}^0[\varphi](\Bx),
\end{split}
\]
for any $\Bx\in \p B_1$. Similarly, one can show that
$$\Scal_{B_1}^0\p_{\nu_1} \Scal_{B_2}^0|_{\p B_2}=\Dcal_{B_1}^{0}\Scal_{B_2}^{0}|_{\p B_2},$$
and
$$
\Scal_{B_1}^{0}(\Kcal_{B_1}^{0})^*|_{\p B_2}+\Scal_{B_2}^0\p_{\nu_2} \Scal_{B_1}^0|_{\p B_2}=\Dcal_{B_1}^{0}\Scal_{B_1}^{0}|_{\p B_2}+\Kcal_{B_2}^0\Scal_{B_1}^0.
$$

\end{document}